\documentclass[a4paper,11pt, leqno]{amsart}
\usepackage{amsmath,amssymb,amsthm,mathtools}
\usepackage{enumitem}
\usepackage{hyperref}
\usepackage{mathrsfs}
\usepackage{tikz}
\usepackage{comment}
\usepackage{todonotes}

\hypersetup{
	colorlinks=true,
	linkcolor=blue,
	citecolor=blue,
	urlcolor=blue
}

\numberwithin{equation}{section}

\theoremstyle{definition}
\newtheorem{definition}{Definition}[section]
\newtheorem{remark}[definition]{Remark}

\newtheorem{question}[definition]{Question}

\theoremstyle{plain}
\newtheorem{theorem}[definition]{Theorem}

\newtheorem{corollary}[definition]{Corollary}
\newtheorem{lemma}[definition]{Lemma}

\newcommand{\C}{\mathbb{C}}

\newcommand{\D}{\mathbb D}

\title[Quasiconformal parametrizations and circular length]
{Length control for foliations by Jordan curves via smooth quasiconformal maps}
\date{}
\author{Abhishek Pandey}
\address{Department of Mathematics and Statistics, University of Jyv\"askyl\"a,
P.O. Box 35, FI-40014 University of Jyv\"askyl\"a, Finland}
\email{pandey.a.pandey@jyu.fi}


\thanks{A. Pandey was supported by the Research Council of Finland, project number 360505.}
\begin{document}

\begin{abstract}
Let $\Omega\subsetneq\mathbb C$ be a non-empty simply connected domain and $0<p<2$. Bouchala and Koskela [Proc. Amer. Math. Soc. (2024)] proved that there exists a quasiconformal map $f:\mathbb D\to\Omega$ such that $\int_0^1\ell(f(r S^1))^p\,dr<\infty$, and asked whether this conclusion holds at $p=2$, or even for every $0<p<2+\varepsilon$ with some $\varepsilon>0$. We answer this question affirmatively with $\varepsilon=2/157$ by constructing a single smooth $5$-quasiconformal map $f:\mathbb D\to\Omega$ satisfying this integrability condition simultaneously for every $0<p<2+2/157$.
\end{abstract}

\maketitle

\section{Introduction}

For $0<p<\infty$, the Hardy space $H^p$ is the set of all analytic functions $f:\D\to\C$ such that
\begin{equation}\label{eq:Hardy-definition}
 \|f\|_{H^p}^p:=\sup_{0<r<1} \int_{S^1}|f(r\omega)|^p d\sigma(\omega) =\sup_{0<r<1}\int_0^{2\pi}|f(re^{i\theta})|^p\,d\theta<\infty.
\end{equation}
Here $S^1=\{z\in\C:|z|=1\}$ and $\sigma$ denotes Euclidean arclength measure on $S^1$, so that $\sigma(S^1)=2\pi$.

Astala and Koskela~\cite{AK} developed the $H^p$-theory of quasiconformal mappings by defining a quasiconformal map $f$ of $\D$ to belong to $H^p$ if \eqref{eq:Hardy-definition} holds.

For a homeomorphism $f:\D\to f(\D)\subset\C$, define
\begin{equation}\label{eq:length-function}
 L_f(r):=\ell\bigl(f(rS^1)\bigr),\quad 0<r<1,
\end{equation}
where $\ell$ denotes Euclidean arc-length, with the value
$+\infty$ allowed.

Bouchala and Koskela proved the following two existence results in
\cite{BK}.

\begin{theorem}\cite[Theorem~1.1]{BK}\label{thm:BK-Hardy}
 Let $\Omega\subsetneq\C$ be a non-empty simply connected domain. Let $p\in (0, \infty)$. Then there is a quasiconformal map $f:\D\to \Omega$ such that $f\in H^p$.
\end{theorem}

\begin{theorem}\label{thm:BK-length}\cite[Corollary~1.2]{BK}
 Let $\Omega\subsetneq\C$ be a non-empty simply connected domain. Let $0<p<2$. Then there is a quasiconformal map $f:\D\to\Omega$ such that
 \begin{equation}\label{eq:BK-length}
  \int_0^1 L_f(r)^p\,dr<\infty.
 \end{equation}
\end{theorem}

Bouchala and Koskela asked the following question:
\begin{question}
Does the claim of Theorem~\ref{thm:BK-length} also hold
for $p=2$, or even for all $0<p<2+\varepsilon$
for some $\varepsilon>0$?
\end{question}
We answer this question affirmatively as follows.

\begin{theorem}\label{thm:main}
Let $\Omega\subsetneq\C$ be a non-empty simply connected domain and let $K\geq1$ be an integer such that $2+2/157<K/2$. Then there exists a $K$-quasiconformal map $f:\D\to\Omega$
such that
\begin{equation}\label{eq:main-length}
  \int_0^1 L_f(r)^p\,dr<\infty,
 \end{equation}
for every $0<p<2+2/157$. Moreover, $f\in H^p$ for every $0<p<K/2$.
In particular, one may take $K=5$.
\end{theorem}

Thus one may take $\varepsilon=2/157$ in the answer to
the question of Bouchala and Koskela. The constant $2/157$ appears in the conformal arc-length theorem of Gonz\'alez and Pel\'aez~\cite{GP}. 

\subsection{Idea of the proof} The proof combines a radial power map with a holomorphic root
of a translated Riemann map. This construction produces a smooth
$5$-quasiconformal parametrization $f$ and a nonvanishing univalent
analytic function $g$ satisfying $L_f(r)\leq5L_g(r)$.
Prawitz's theorem gives $g\in H^p$ for every $0<p<5/2$.
The desired length integrability then follows from the
arc-length characterization of Gonz\'alez and Pel\'aez.

\section{Preliminaries}
Throughout, $\|A\|$ denotes the operator norm of a real linear
map $A$.

\begin{definition}
Let $K\geq1$, and let $\Omega,\Omega'\subset\C$ be domains.
An orientation-preserving homeomorphism $f:\Omega\to\Omega'$
is called $K$-quasiconformal if
$f\in W^{1,2}_{\mathrm{loc}}(\Omega;\C)$ and
\begin{equation}\label{eq:qc-definition}
    \|Df(z)\|^2\leq KJ_f(z)
\end{equation}
for almost every $z\in\Omega$, where $Df$ denotes the differential of $f$ and $J_f$ is its Jacobian.
\end{definition}

For an analytic function $f$ on $\D$, define
\begin{equation}\label{eq:integral-mean}
    M_s(r,f)
    :=\left(\frac{1}{2\pi}\int_0^{2\pi}
    |f(re^{i\theta})|^s\,d\theta\right)^{1/s},
    \quad 0<s<\infty,\quad 0<r<1,
\end{equation}
and
\begin{equation}\label{eq:maximum-modulus}
    M_\infty(r,f):=\max_{|z|=r}|f(z)|,
    \quad 0<r<1.
\end{equation}

If $f$ is also univalent, then
\begin{equation}\label{eq:conformal-length}
    L_f(r)
    =\int_0^{2\pi}r|f'(re^{i\theta})|\,d\theta
    =2\pi rM_1(r,f'),
    \quad 0<r<1.
\end{equation}
In particular, for every $p>0$,
\begin{equation}\label{eq:r-in-1/2-to-1}
    M_1(r,f')^p\leq\pi^{-p}L_f(r)^p,
    \quad \frac12\leq r<1.
\end{equation}

We shall use the following two results about univalent
functions.

\begin{theorem}[Prawitz's theorem {\cite[Chapter~3]{Duren}}]
\label{thm:Prawitz's-theorem}
Every analytic and univalent function on $\D$ belongs to
$H^p$ for every $0<p<1/2$. The Koebe function shows that
the exponent $1/2$ is sharp.
\end{theorem}

Gonz\'alez and Pel\'aez proved the following arc-length
characterization~\cite[Theorem~1]{GP}; see also
\cite[(1.5) and the subsequent discussion]{PRV} for the
precise range of exponents.

\begin{theorem}[Gonz\'alez--Pel\'aez]\label{thm:GP}
Let $f$ be analytic and univalent on $\D$. If
$0<p<2+2/157$, then
\begin{equation}\label{eq:GP}
    \|f\|_{H^p}^p
    \asymp_p
    |f(0)|^p+\int_0^1M_1(r,f')^p\,dr.
\end{equation}
Here $\asymp_p$ denotes two-sided comparability with
positive constants depending only on $p$.
\end{theorem}

\begin{corollary}\label{cor:conformal-length-Hardy}
Let $f$ be analytic and univalent on $\D$, and let
$0<p<2+2/157$. Then
\begin{equation}\label{eq:length-equivalence}
    f\in H^p
    \quad\Longleftrightarrow\quad
    \int_0^1L_f(r)^p\,dr<\infty.
\end{equation}
\end{corollary}

\begin{proof}
Suppose that $f\in H^p$. By
\eqref{eq:conformal-length} and Theorem~\ref{thm:GP},
$$
    \int_0^1L_f(r)^p\,dr
    \leq (2\pi)^p\int_0^1M_1(r,f')^p\,dr
    <\infty.
$$

Conversely, suppose that $\int_0^1L_f(r)^p\,dr<\infty$.
By \eqref{eq:r-in-1/2-to-1},
$$
    \int_{1/2}^1M_1(r,f')^p\,dr
    \leq\pi^{-p}\int_{1/2}^1L_f(r)^p\,dr
    <\infty.
$$
Since $f'$ is continuous on the compact disk
$\{z:|z|\leq1/2\}\subset\D$, the constant
$M:=\max_{|z|\leq1/2}|f'(z)|$ is finite.
Consequently, $M_1(r,f')\leq M$ for $0<r\leq1/2$.
Thus $\int_0^1M_1(r,f')^p\,dr<\infty$, and
Theorem~\ref{thm:GP} gives $f\in H^p$.
\end{proof}

\section{Proof of Theorem~\ref{thm:main}}

We first give the radial power map construction.

\begin{lemma}\label{lem:radial-map}
 Fix $a\in\C$ and an integer $K\geq1$. Define $\mathcal{R}_{K,a}:\C\to\C$ by
 \begin{equation}\label{eq:radial-map}
 \mathcal{R}_{K,a}(z)
 =
 \begin{cases}
  a+(z-a)|z-a|^{\frac{1}{K}-1}, & z\neq a,\\
  a, & z=a.
 \end{cases}
\end{equation}
Then $\mathcal{R}_{K,a}$ is a $K$-quasiconformal map.   
\end{lemma}

\begin{proof}
The inverse is given by
$$
 \mathcal{R}_{K,a}^{-1}(w)
 =
 \begin{cases}
  a+(w-a)|w-a|^{K-1}, & w\neq a,\\
  a, & w=a.
 \end{cases}
$$
Both $\mathcal{R}_{K,a}$ and its inverse are continuous.
Thus $\mathcal{R}_{K,a}$ is a homeomorphism. Write $\alpha=1/K$. For $z\neq a$, the singular values of
$D\mathcal{R}_{K,a}(z)$ in the radial and tangential directions
relative to $a$ are
\begin{equation}\label{eq:radial-singular-values}
 \lambda_r(z)=\alpha|z-a|^{\alpha-1},
 \quad
 \lambda_\theta(z)=|z-a|^{\alpha-1}.
\end{equation}
Since $0<\alpha\leq1$, it follows that
$\|D\mathcal{R}_{K,a}(z)\|=|z-a|^{\alpha-1}$ and
\begin{equation}\label{eq:radial-distortion}
 J_{\mathcal{R}_{K,a}}(z)
 =\alpha|z-a|^{2\alpha-2}>0,
 \quad
 \frac{\|D\mathcal{R}_{K,a}(z)\|^2}
 {J_{\mathcal{R}_{K,a}}(z)}
 =\frac1\alpha=K.
\end{equation}

To verify Sobolev regularity at $a$, observe that, for every
$\delta>0$,
\begin{equation}\label{eq:radial-W12}
 \int_{B(a,\delta)\setminus\{a\}}
 \|D\mathcal{R}_{K,a}(z)\|^2\,dA(z)
 =
 2\pi\int_0^\delta r^{2\alpha-1}\,dr<\infty.
\end{equation}
Since $|\mathcal{R}_{K,a}(z)-a|=|z-a|^\alpha$,
integration by parts on punctured balls, applied to
$\mathcal{R}_{K,a}-a$, gives an inner boundary term of order
$O(\varepsilon^{\alpha+1})$.
Thus the classical derivatives extend as weak derivatives
across $a$, and
$\mathcal{R}_{K,a}\in W^{1,2}_{\mathrm{loc}}(\C;\C)$.
Together with \eqref{eq:radial-distortion}, this proves that
$\mathcal{R}_{K,a}$ is a $K$-quasiconformal map.
\end{proof}

\begin{theorem}\label{thm:transfer}
 Let $\Omega\subsetneq\C$ be a non-empty simply connected domain, let $a\in\C\setminus\Omega$, and let $K\geq1$ be an integer. There exist a $K$-quasiconformal map $f:\D\to\Omega$ and a nonvanishing univalent analytic function $g$ on $\D$ such that
 \begin{align}
  |f(z)-a|&=|g(z)|,\quad z\in \D\label{eq:modulus-identity}\\
  L_f(r)&\leq KL_g(r),\quad 0<r<1, \label{eq:length-transfer}
 \end{align}
 and 
 \begin{equation}\label{eq:transfer-Hardy}
 g\in H^p \quad\text{and}\quad f\in H^p
  \quad\text{for every } 0<p<\frac{K}{2}.
 \end{equation}
\end{theorem}

\begin{proof}
Let $K\ge 1$ and $a\notin \Omega$. Let $F=\mathcal{R}_{K,a}:\C\to \C$ be the $K$-quasiconformal map from Lemma~\ref{lem:radial-map}. Consider the non-empty proper simply connected domain $D:=F^{-1}(\Omega)$ and let $\varphi:\D\to D$ be the conformal map from the Riemann mapping theorem. Observe that $a\notin D$ since $F(a)=a$ and $a\notin \Omega$. Therefore the analytic function $\varphi-a$ is nonvanishing on $\D$. Hence there exists an analytic function $h$ on $\D$ such that $e^{h(z)}=\varphi(z)-a$. Define $g(z):=e^{h(z)/K}$. Then $g$ is analytic, nonvanishing, and univalent since $\varphi$ is univalent. Moreover,
\begin{equation}\label{eq:relation-of-g}
    g^K=e^h=\varphi-a
\end{equation}

Define $f:\D\to \Omega$ by $f:=F\circ \varphi$. Then by Lemma~\ref{lem:radial-map} $f$ is a $K$-quasiconformal map. Moreover,
\begin{equation}\label{eq:f-via-g}
 f-a
 =(\varphi-a)|\varphi-a|^{1/K-1}
 =\frac{g^K}{|g|^{K-1}}.
\end{equation}

Taking absolute values proves \eqref{eq:modulus-identity}. 

Observe that $f-a=T\circ g$, where $T:\C\setminus\{0\}\to\C\setminus\{0\}$ is defined by
\begin{equation}\label{eq:T-K}
 T(w):=\frac{w^K}{|w|^{K-1}}.
\end{equation}
For $w=\rho e^{i\theta}$, we have
$T(w)=\rho e^{iK\theta}$. Thus the radial and tangential singular values of $DT(w)$ are $1$ and $K$, respectively, and consequently
\begin{equation}\label{eq:T-derivative}
 \|DT(w)\|=K,
 \quad w\neq0.
\end{equation}

Since $g$ is nonvanishing, both $g$ and $f$ are smooth on $\D$. Applying the chain rule to $f-a=T\circ g$ and using \eqref{eq:T-derivative}, we get

\begin{align}
\nonumber \left|\frac{\partial}{\partial\theta}f(re^{i\theta})\right|
 &=
 \left|
 DT\bigl(g(re^{i\theta})\bigr)
 \left[\frac{\partial}{\partial\theta}g(re^{i\theta})\right]
 \right|\\ \nonumber
 &\leq
 \left\|DT\bigl(g(re^{i\theta})\bigr)\right\|
 \left|\frac{\partial}{\partial\theta}g(re^{i\theta})\right|\\
 &\leq
 K\left|\frac{\partial}{\partial\theta}g(re^{i\theta})\right|,\label{eq:angular-comparison}
\end{align}
for every $0<r<1$ and $\theta\in[0,2\pi]$. For each $0<r<1$, the smooth parametrization $\theta\mapsto f(re^{i\theta})$ traverses $f(rS^1)$ exactly once. Hence the arc-length formula and
\eqref{eq:angular-comparison} give
$$
 L_f(r)
 =
 \int_0^{2\pi}
 \left|\frac{\partial}{\partial\theta}f(re^{i\theta})\right|
 \,d\theta
 \leq
 K\int_0^{2\pi}
 \left|\frac{\partial}{\partial\theta}g(re^{i\theta})\right|
 \,d\theta
 =
 KL_g(r),
$$
which proves \eqref{eq:length-transfer}.

If $0<p<K/2$, then $p/K<1/2$, so
Theorem~\ref{thm:Prawitz's-theorem} gives
$\varphi-a\in H^{p/K}$.
Since $|g(z)|^p=|\varphi(z)-a|^{p/K}$ for every $z\in\D$,
we obtain $g\in H^p$.
By \eqref{eq:modulus-identity}, $f-a\in H^p$, and hence
$f\in H^p$.
\end{proof}

\begin{proof}[Proof of Theorem~\ref{thm:main}]
Fix an integer $K\geq1$ such that $2+2/157<K/2$,
and choose $a\in\C\setminus\Omega$.
Theorem~\ref{thm:transfer} gives maps $f$ and $g$
satisfying \eqref{eq:length-transfer} and
\eqref{eq:transfer-Hardy}. For every $0<p<2+2/157$,
Corollary~\ref{cor:conformal-length-Hardy} and
\eqref{eq:length-transfer} give
$$
\int_0^1L_f(r)^p\,dr
\leq K^p\int_0^1L_g(r)^p\,dr<\infty.
$$
The asserted Hardy membership follows from
\eqref{eq:transfer-Hardy}.
In particular, one may take $K=5$.
\end{proof}

\begin{remark}\label{rem:transfer-principle}
Note that Theorem~\ref{thm:transfer} recovers Theorem~\ref{thm:BK-Hardy}; for a given $p>0$, choose an integer $K>2p$. Moreover, fix $p>0$ and suppose that the following implication holds
for every analytic and univalent function $F$ on $\D$.
\begin{equation}\label{eq:main-implication}
F\in H^p
\quad\Longrightarrow\quad
\int_0^1L_F(r)^p\,dr<\infty.
\end{equation}
Then, for every non-empty proper simply connected domain
$\Omega\subsetneq\C$ and every integer $K>2p$,
Theorem~\ref{thm:transfer} gives a $K$-quasiconformal
homeomorphism $f:\D\to\Omega$ satisfying
$$
\int_0^1L_f(r)^p\,dr<\infty.
$$

More generally, suppose \eqref{eq:main-implication} holds
for every $p$ in an interval $(0,P)$. If an integer $K$ is chosen so that $K>2P$, then the single $K$-quasiconformal map \textcolor{blue}{given} by
Theorem~\ref{thm:transfer} satisfies
$$
\int_0^1L_f(r)^p\,dr<\infty
\quad\text{for every }0<p<P.
$$
Thus any extension of the exponent range in Theorem~\ref{thm:GP} yields a corresponding extension
of the exponent range in Theorem~\ref{thm:main}.
\end{remark}

\begin{remark}
The map $f$ constructed in the proof of
Theorem~\ref{thm:main} is a smooth diffeomorphism.
Consequently, the curves $\Gamma_r=f(rS^1)$, $0<r<1$,
form a smooth foliation of $\Omega\setminus\{f(0)\}$, being the level sets of the smooth submersion
$u:\Omega\setminus\{f(0)\}\to(0,1)$ given by
$u(w)=|f^{-1}(w)|$. Their lengths satisfy
$\int_0^1\ell(\Gamma_r)^p\,dr<\infty$
for every exponent $p$ in the range of the theorem.
Adjoining the point leaf $\{f(0)\}$ gives a singular
foliation of $\Omega$ with a single singular point.
\end{remark}

\section*{Acknowledgements}
The author thanks Pekka Koskela for suggesting the problem and for valuable discussions and insights, including the interpretation of the construction as a foliation by Jordan curves.

\bibliographystyle{amsalpha}
\bibliography{bibliography}

\end{document}